\documentclass[preprint,aos]{imsart}
\setattribute{journal}{name}{}
\usepackage[colorlinks,citecolor=blue,urlcolor=blue]{hyperref}

\usepackage{amssymb}
\usepackage{amsmath}
\usepackage{amsthm}
\usepackage{enumerate}
\usepackage[authoryear]{natbib}
\usepackage{dsfont}
\usepackage{mathrsfs}
\usepackage{color}

\def\quotation#1{``#1''}

\newcommand{\PP}{\mathbb{P}}
\newcommand{\EE}{{\mathbb{E}}}

\newcommand{\lle}{\, \, {\lesssim}\, \, }

\newcommand{\gge}{\, \, {\gtrsim}\, \, }

\newcommand{\eps}{\varepsilon}

\newcommand{\RR}{\mathbb{R}}

\newcommand{\NN}{\mathbb{N}}

\newcommand{\qv}[1]{\left< #1 \right>}

\newcommand{\jjsumN}{\sum_{j=0}^{N-1}}
\newcommand{\jjsumn}{\sum_{j=0}^{n-1}}
\newcommand{\sumin}{\sum_{i=1}^n}

\newtheorem{thm}{Theorem}[section]
\newtheorem{lem}{Lemma}[section]

\newtheorem*{assumption}{Assumption}
\theoremstyle{definition}

\begin{document}

\begin{frontmatter}

\title{\bf Minimax lower bounds for function estimation on graphs}

\author{Alisa Kirichenko and Harry van 
Zanten
}

\maketitle

\bigskip

\begin{abstract}
We study minimax lower bounds for function estimation problems on large graph when the target function is smoothly varying over the graph. We derive minimax rates in the context of regression and classification problems on graphs that satisfy an asymptotic shape assumption and with a smoothness condition on the target function, both 
formulated in terms of the graph Laplacian.

\end{abstract}

\end{frontmatter}

\numberwithin{equation}{section}

\numberwithin{equation}{section}

\section{Introduction}

In recent years there has been substantial interest in high-dimensional 
estimation and prediction problems on large graphs. These can in many cases be seen 
as high-dimensional or nonparametric regression or classification problems
in which the goal is to learn a ``smooth'' function on a given graph. 
Various methods have been proposed to deal with such problems, motivated 
by a variety of applications. 
Any sensible method employs some form of regularisation that takes 
the geometry of the graph into account. Examples of methods that have
been considered include penalised least squares regression using a Laplacian-based penalty
(e.g.\ \cite{ando2007, belkin2004, kolaczyk2009, smola2003, zhu2005}), penalisation using the total variation norm (e.g.\ \cite{ryan}) 
and Bayesian regularisation (e.g.\ \cite{jarno}, \cite{andrew}, \cite{me}).

There exist only a few papers that study theoretical aspects of the performance of 
nonparametric estimation procedures on graphs. Early references are \cite{belkin2004}, in which a theoretical analysis of a Tikhonov regularisation method 
is conducted in terms of algorithmic stability and
\cite{johnson2007effectiveness} and \cite{ando2007}, who consider sub-sampling schemes for estimating 
a function on a graph. 
More recently, convergence rates have been obtained by \cite{ryan} in the context of 
regression on a regular grid using total variation penalties and by 
\cite{me} for nonparametric Bayes procedures for regression and classification 
on more general graphs. 
The paper \cite{ryan} also establishes minimax lower bounds for regression problems on grids.

In this paper we derive new minimax results for regression and binary classification on graphs, 
exhibiting the best possible rates that can be attained uniformly over certain 
classes of ``smooth'' functions on graphs. We consider simple undirected graphs that satisfy an assumption on their ``asymptotic geometry'', formulated in terms of the 
graph Laplacian. This assumption, which is recalled in the next section, 
was introduced in \cite{me}. 
The geometry assumption attaches a parameter $r$ to a graph which essentially describes 
how the eigenvalues of its Laplacian behave.  It was illustrated in \cite{me}
that it  is satisfied for many graphs of interest. Theoretically it can be shown to 
hold for instance for regular grids and tori  of arbitrary dimensions. 
Moreover, for a given graph it can be verified empirically whether the assumption 
is reasonable and what the corresponding graph parameter $r$ is. 
 In \cite{me} this was done both for simulated ``small world graphs''  
and for real protein-protein interaction graphs.

The geometry parameter $r$ appears in the minimax 
lower bounds we derive. The other key ingredient is the regularity $\beta$ of 
the function that is being estimated, defined in a suitable manner. 
 We introduce a Sobolev-type smoothness condition on the target function using the 
 graph Laplacian again to quantify smoothness. 
The geometry of the underlying graph and the smoothness of the target function together 
determine the minimax rate through the geometry parameter $r$ and the smoothness parameter $\beta$. 
We have chosen our setup and normalisations in such a way that the optimal rates 
over balls of smooth functions that we obtain are of the usual form $n^{-\beta/(r+2\beta)}$. 
This shows that the geometry parameter $r$ can be interpreted as some 
kind of ``dimension'' of the graph. In the regular grid case it is indeed precisely the dimension 
of the grid, as shown in \cite{me}. However, the result holds for 
much more general graphs as well. In particular, the geometry parameter $r$ does not need 
to be an integer.

For the sake of completeness we give two-sided results, that is, we also 
exhibit estimators that achieve the lower bounds, showing that the bounds are tight.
However, these estimators are non-adaptive, in the sense that they depend on the
smoothness parameter $\beta$, which will typically not be accessible in realistic settings. 
More interestingly perhaps, the lower bounds match the upper bounds 
we obtained in \cite{me}. This shows that the nonparametric Bayes procedures we 
proposed in the latter paper are smoothness-adaptive and rate-optimal. 
We note however that the  procedure exhibited in \cite{me} that is adaptive on the whole
range of regularities $\beta> 0$ is only rate-optimal up to a logarithmic factor. 
It might be of interest to study the possibility of procedures that achieve exactly the correct rate.

In the next section we introduce the general framework. Specifically, we define the geometry condition on the graph and the smoothness condition on the target function. In Section \ref{main} we describe the regression and classification problems on a graph and present our results on 
the minimax rates for those problems. The mathematical proofs are given in Section \ref{proofs}. 

\section{Setting} \label{sec: setting}
Let $G$ be a connected simple undirected graph with vertices labelled $\{1,\dots,n\}$. 
Let $A$ be its adjacency matrix, i.\,e. $A_{ij}$ is $1$ or $0$ according 
to whether or not there is an edge between vertices $i$ and $j$. Let $D$ be the diagonal matrix 
with element $D_{ii}$ equal to the degree of vertex $i$. Let $L = D-A$ be the Laplacian of 
the graph. 

A function $f$ on the (vertices of the) graph is simply a function $f: \{1, \ldots, n\} \to \RR$. 
We measure distances and norms of functions using the norm $\|\cdot\|_n$ defined by 
$\|f\|^2_n = n^{-1} \sum_{i=1}^n f^2(i)$.
The corresponding inner product of two functions $f$ and $g$ is denoted by 
\[
\qv{f,g}_n = \frac1n \sum_{i=1}^n f(i)g(i).
\]
The Laplacian is nonnegative definite (\cite{thebook}). Hence
we can order the Laplacian eigenvalues
by magnitude and denote them by
\[
0 = \lambda_{n, 0} < \lambda_{n,1} \le \lambda_{n,2} \le \cdots \le \lambda_{n,n-1}.
\]
(The smallest one always equals $0$ and since the graph is connected, the second
one is positive, see for instance \cite{thebook}).
We fix a corresponding sequence of eigenfunctions $\psi_{i}$, orthonormal with respect to the inner product $\qv{\cdot, \cdot}_n$. 

We derive our results under an asymptotic geometry 
assumption on the graph, first introduced in \cite{me}, 
formulated in terms of the Laplacian eigenvalues.

\bigskip

\begin{assumption}
We say that the {\em geometry assumption is satisfied with parameter $ r \ge 1$} 
if there exist $i_0 \in \NN$, $\kappa \in (0,1]$ and $C_1, C_2 > 0$ such that 
for all $n$ large enough, 
\[
C_1\Big(\frac i n\Big)^{2/ r} \le \lambda_i \le C_2\Big(\frac i n\Big)^{2/ r}, 
 \qquad \text{for all $i \in \{i_0, \ldots, \kappa n\}$}.
\]
\end{assumption}

\bigskip
 
Very roughly speaking the condition  means that asymptotically, or from ``far away'', the graph looks like an $r$-dimensional grid with $n$ vertices.
From \cite{me} we know the assumption is satisfied for $d$-dimensional grids with $r$ equal to the dimension $d$, hence our results on the minimax rates include the usual statements for regression and classification with regular, fixed design. We stress however that the constant $r$ does not need to be a natural number. For given graphs the parameter $r$ can be calculated numerically. For example, in \cite{me} a Watts-Strogatz \quotation{small world} graph is considered which satisfies the condition with $r$ equal to $1.4$. 
Observe that we do not assume the existence of a \quotation{limiting manifold} for the graph as $n\to\infty.$ See \cite{me} for more discussion of the geometry assumption and more examples.

We describe the smoothness of the function of interest by assuming it belongs 
to a Sobolev-type ball of the form 
\begin{equation}
\label{eq: h}
H^\beta(Q) = \Big\{f: \qv{f, (I+(n^\frac{2}{r} L)^\beta) f}_n \le Q^2\Big\}
\end{equation}
for some $\beta, Q > 0$ (independent of $n$). 
This should be viewed as the natural discrete graph version 
of the usual notion of a Sobolev ball of functions on $[0,1]^r$. 
(This is most easily seen in the case of the path graph and $r=1$, 
as illustrated in Example 3.1 of \cite{me}.)
The particular normalisation, which depends on the 
geometry parameter $r$, ensures non-trivial asymptotics. 
Again, we stress that we do not assume that the functions on the graph 
are discretised versions of certain continuous objects on a ``limiting manifold''.

\section{Main results}
\label{main}

Now that we have introduced ways to quantify the graph geometry 
and the regularity of the target function, we can formulate
our minimax results for regression and classification. 
In both cases, $G = G_n$ will be a connected simple undirected graph with vertices 
$1,\dots,n$, satisfying the geometry assumption for $r\geq 1$. 
The target function will be a regression function that is observed with additive Gaussian 
noise in the regression case and a binary regression function in the classification case.

In the regression case we assume that we have observations $Y = (Y_1,\dots,Y_n)$ at the vertices of the graph satisfying
\begin{equation}
\label{model}
Y_i=f(i)+\sigma \xi_i, \qquad i =1, \ldots, n, 
\end{equation}
where the $\xi_i$ are independent standard Gaussians, $\sigma>0$ and $f: \{1, \ldots, n\} \to \RR$ is the unknown function of interest. We denote the corresponding distribution of $Y$ by $\PP_{f}$ 
and the associated expectation by $\EE_{f}$. 
Our main result in this setting is the following.\footnote
{We write $a_n \asymp b_n$ if $0 < \liminf a_n/b_n \le \limsup a_n/b_n < \infty$.}

\bigskip

\begin{thm}[Regression]
\label{general_thm_reg}
Suppose that the graph satisfies the geometry assumption for $r\geq1$.
Then for all $\beta, Q >0$
\begin{equation}
\label{minimax_rate}
\inf_{{\hat f}}\sup_{f\in H^\beta(Q)} \EE_f \|{\hat f}-f\|^2_n \asymp n^{-2\beta/(2\beta+r)},
\end{equation}
where the infimum is taken over all estimators $\hat f = \hat f(Y_1, \ldots, Y_n)$.
\end{thm}

\bigskip

The theorem shows that the minimax rate for the regression problem on the graph is equal to $n^{-\beta/(2\beta+r)}$. We obtain the upper bound on the rate by constructing a projection estimator $\tilde f$ for which
\[
\sup_{f\in H^\beta(Q)} \EE_f \|{\tilde f}-f\|^2_n \lle n^{-2\beta/(2\beta+r)}.
\]
The proof shows that this estimator depends on the regularity level of the target function and therefore is not adaptive. An adaptive (Bayesian) procedure is exhibited in 
 \cite{me}.

In the binary classification case we assume that the data $Y_1, \ldots, Y_n$ are 
independent $\{0, 1\}$-valued variables, observed at the vertices of the graph. 
In this case the goal is to estimate the binary regression function $\rho$, or ``soft label function'' on the 
graph, defined by 
\[
\rho(i) = \PP(Y_i = 1). 
\]
The function $\rho$, of course, determines the distribution of the data, which we therefore 
denote by $\PP_\rho$. Again, the associated expectations are denoted by $\EE_\rho$. 

Technically the classification case is slightly more demanding. 
Different from the regression case we also have to impose conditions on the 
Laplacian eigenfunctions $\psi_j$ in this case. Moreover, 
we impose the regularity condition not directly on the binary regression function $\rho$, 
but on a suitably transformed version of it, so that it maps into $\RR$ instead of $(0,1)$. 
Concretely, we fix a differentiable link function $\Psi: \RR \to (0, 1)$
such that 
${\Psi'}/({\Psi(1-\Psi)})$ is uniformly bounded, and $\Psi' > 0$ everywhere. 
Note that for instance the 
sigmoid, or logistic link $\Psi(f) = 1/(1+\exp(-f))$ satisfies this condition. 
Under these conditions the inverse $\Psi^{-1}: (0,1)\to \RR$ is well defined. In the classification setting the regularity condition will be formulated in terms of $\Psi^{-1}(\rho)$.

Recall from Section \ref{sec: setting} that $\psi_j$ is a sequence of eigenfunctions (or vectors) 
of the graph Laplacian $L$ that is orthonormal with respect to the norm $\|\cdot\|_n$. 
In particular, 
the eigenfunctions $\psi_j$ are normalised such that 
\[
\frac1n \sum_{i=1}^n\psi_j^2(i) = 1
\]
for every $j$. In the following theorem we assume in addition that they are uniformly
bounded by a common constant $C>0$, which is independent of $n$, i.e.\
that $|\psi_j(i)| \le C$ for every $i$ and $j$. 
We need this technical assumption in the proof of the lower bound.

\bigskip

\begin{thm}[Classification]
\label{general_thm_cl}
Suppose that the graph satisfies the geometry assumption for $r\geq1$
and that the Laplacian eigenfunctions are uniformly bounded by a common constant $C > 0$, independent of $n$. 
 Let $\Psi: \RR \to (0,1)$ be a differentiable link function as above with inverse $\Psi^{-1}$. 
 Then for $\beta\geq r/2$ and $Q > 0$
\[
\inf_{{\hat \rho}} \sup_{\rho: \Psi^{-1}(\rho)\in H^{\beta}(Q)} \EE_\rho \|{\hat \rho}-\rho\|^2_n\asymp n^{-2\beta/(2\beta+r)},
\]
where the infimum is taken over all estimators $\hat\rho = \hat\rho(Y_1, \ldots, Y_n)$.
\end{thm}

\bigskip

Also in this case the proof of the theorem provides an estimator that 
achieves the lower bound, but that estimator depends on $\beta$ and hence
is non-adaptive. Adaptive, rate-optimal (Bayes) procedures have been exhibited 
for this classification setting as well in \cite{me}. 
Note that compared to the regression case, there is an extra technical requirement $\beta\geq r/2$. 
Additionally, we assume boundedness of the Laplacian eigenfunctions. 
In principle, for a specific case this can be verified numerically. 
For regular grids of arbitrary dimensions it is straightforward to see 
that this condition is fulfilled.

Indeed, by \cite{eigenvectors}, for instance, the (unnormalised) eigenvectors of the graph Laplacian of the path graph are given by 
\[
\tilde\psi_j(i)= \cos\left(\pi ij/n-\pi j/2n\right).
\] 
For the $\|\cdot\|_n$-norm of the $j$th eigenvector we then have 
\begin{align*}
\|\tilde\psi_j\|_n^2 & =\frac 1n\sum_{i=1}^{n}\cos^2\left(\pi ij/n-\pi j/2n\right)=\\
& =\frac12+\frac 1{4n\sin(\pi j/n)}\sumin2\sin(\pi j/n)\cos((2i-1)\pi j /n).
\end{align*}
By well known trigonometric identities we have, for any $x\in \RR$,
\[
\sumin 2\sin x \cos(2ix-x)=\sumin(\sin 2ix- \sin (2ix-2x))=\sin 2 n x.
\]
It follow that for any $j=1,\dots,n-1$, 
\[
\|\tilde\psi_j\|_n^2=\frac 12+ \frac 1{4n} \frac{\sin2\pi j}{\sin\pi j/n}=\frac12.
\]
Notice that $\|\tilde\psi_0\|_n^2=1.$ 
So we see that indeed, the normalised eigenvectors $\psi_j$ of the Laplacian of the path graph 
are uniformly bounded by a common constant.
The eigenvectors of a Cartesian product of two graphs are equal to the Kronecker products of pairs of eigenvectors associated with the Laplacians of those graphs. Since the grids of higher dimensions are products of path graphs, they satisfy the condition as well.

\section{Proofs}
\label{proofs}

\subsection{Proof of Theorem \ref{general_thm_reg}}
\label{proof_reg}

In the regression case we first expand the observations in the 
eigenbasis of the graph Laplacian, which brings the problem into the setting of the 
sequence formulation of the white noise model. 
Then we adapt techniques from the proof of Pinsker's theorem as can it is given, 
for instance, in \cite{tsybakov}.

\subsubsection{Preliminaries}

Let $Y=(Y_1,\dots,Y_n)$, $\xi=(\xi_1,\dots,\xi_n)$, and let $\psi_i$ be the orthonormal eigenfunctions of the graph Laplacian. Denote $\tilde\xi_i=\langle\xi, \psi_i\rangle_n $ and observe that the $\tilde \xi_i$ are centred Gaussian with
\[
\EE \tilde\xi_i\tilde\xi_j= \frac 1n \delta_{ij}.
\]
The inner products $Z_i=\langle Y, \psi_i\rangle_n$ satisfy the following relation for $i=0,\dots, n-1$
\[
Z_i= \langle Y, \psi_i\rangle_n=f_i+ \sigma \tilde\xi_i,
\]
where $f_i$ are coefficients in the decomposition of the target function $
f_0=\sum_{i=0}^{n-1} f_i\psi_i.
$
Additionally, consider the decomposition of an estimator $\hat f=\sum_{i=0}^{n-1} \hat f_i \psi_i$. Then
\[
\|\hat f- f_0\|^2_n= \left\langle \sum_{i=0}^{n-1} (\hat f_i-f_i)\psi_i, \sum_{i=0}^{n-1} (\hat f_i-f_i)\psi_i\right\rangle_n=\sum_{i=0}^{n-1} (\hat f_i-f_i)^2.
\] 

Hence, the minimax rates for the original problem are of the same order as the minimax rates for the problem of recovering $f=(f_0,\dots,f_{n-1})$, given the observations
\begin{equation}
\label{model1}
Z_i=f_i+\eps\zeta_i,
\end{equation}
where $\zeta_i$ are independent standard Gaussian and $\eps=\frac \sigma {\sqrt{n}}$. To avoid confusion we define general ellipsoids on the space of coefficients for an arbitrary sequence $a_j> 0$ and some finite constant $Q>0$
\begin{equation}
\label{ellips}
B_n(Q)=\{ f\in\RR^n: \jjsumn a_j^2 f_j^2\leq Q\}.
\end{equation}
For a function $f$ in the Sobolev-type ball $H^\beta(Q)$ its vector of coefficients belongs to $B_n(Q)$ with 
\begin{align*}
a_j^2=&1+ \lambda_j^{2\beta/r}n^{2\beta/r}, \,j=0,\dots,n-1.
\end{align*}

In order to prove the theorem it is sufficient to show that
\begin{equation}
\label{eq: statement}
\inf_{{\hat f}}\sup_{f\in B_n(Q)} \EE_f \left(\sum_{i=0}^{n-1} (\hat f_i-f_i)^2\right) \asymp n^{-2\beta/(2\beta+r)}.
\end{equation}

We are going to follow the proof of Pinsker's theorem (see for example \cite{tsybakov}) which studies a similar case in the setting of the Gaussian white noise model on the interval $[0,1]$. The proof requires some modifications arising from the nature of our problem. The main differences 
with the usual lower bound result over Sobolev balls in the infinite sequence model 
 are that we only have $n$ observations and that our 
 ellipsoids $B_n(Q)$ have a special form.

 In order to proceed we first consider the problem of obtaining minimax rates in the class of linear estimators. We introduce Pinsker's estimator and recall the linear minimax lemma showing that Pinsker's estimator is optimal in the class of linear estimators. The risk of a linear estimator $\hat f(l)=(l_0 Z_0,\dots, l_{n-1}Z_{n-1})$ with $l=(l_0,\dots, l_{n-1})\in\RR^n$ is given by
\[
R(l, f)=\EE_ f\jjsumn(\hat f_j- f_j)^2=\jjsumn\left((1-l_j)^2f_j^2+\eps^2l_j^2\right).
\]
For large $n$ we introduce the following equation with respect to the variable $x$
\begin{equation}
\label{kappaeq}
\frac{\eps^2}{x}\jjsumn a_j(1-x a_j)_+=Q.
\end{equation}
Suppose, there exists a unique solution $x$ of (\ref{kappaeq}). For such a solution, define a vector of coefficients $l'$ consisting of entries
\begin{equation}\label{eq: lprime}
l'_j=(1-x a_j)_{+}.
\end{equation}
The linear estimator $\tilde f=\tilde{f}(l')$ is called the Pinsker estimator for the general ellipsoid $B_n(Q)$. The following lemma, which appears as Lemma 3.2 in \cite{tsybakov}, shows that the Pinsker estimator is a linear minimax estimator. 

\begin{lem}[Linear minimax lemma]
\label{linear_minimax}
Suppose that $B_n(Q)$ is a general ellipsoid defined by (\ref{ellips}) with $Q>0$ and a positive set of coefficients $\{a_j\}_{j=0}^{n-1}$. Suppose there exists a unique solution $x$ of (\ref{kappaeq}) and suppose that the associated coefficients $l'_j$ defined by \eqref{eq: lprime} satisfy
\begin{equation}
\label{eq: s}
S=\eps^2\jjsumn l'_j<\infty.
\end{equation}
Then the linear minimax risk satisfies
\begin{equation}
\inf_{l\in\RR^n}\sup_{f\in B_n(Q)}R(l,f) =\sup_{f\in B_n(Q)} R(l',f)=S.
\end{equation}
\end{lem}

In order to be able to apply Lemma \ref{linear_minimax} in our graph setting 
we need the following technical lemma.

\begin{lem}
\label{pinsker_const}
Consider the ellipsoid $ B_n(Q)$ defined by (\ref{ellips}) with $Q>0$ and 
\begin{align*}
a_j^2=&1+ \lambda_j^{2\beta/r}n^{2\beta/r}, \,j=0,\dots,n-1.
\end{align*}
 Then, as $n\to\infty$, we have the following
\begin{enumerate}[(i)]
\item There exists a solution $x$ of (\ref{kappaeq}) which is unique and satisfies 
\[
x\asymp n^{-\beta/(2\beta+r)}.
\]
\item The associated sum \eqref{eq: s} of the coefficients of the Pinsker's estimator satisfies
\[
S\asymp n^{-2\beta/(2\beta+r)}.
\]
\item For $\eps=\frac{\sigma}{\sqrt n}$ define $v_j=\frac{\eps^2(1-xa_j)_{+}}{xa_j}$. Then 
\[
\max_{j=0,\dots,n-1}v_j^2 a_j^2=O\left(n^{-r/(2\beta+r)}\right).
\]
\end{enumerate}
\end{lem}
\begin{proof}

\begin{enumerate}[(i)]
\item

According to Lemma 3.1 from \cite{tsybakov} for large enough $n$ and for an increasing positive sequence $a_j$ with $a_n\to+\infty$, as $n\to\infty$, there exists a unique solution of (\ref{kappaeq}) given by
\[
x=\frac{\eps^2\jjsumN a_j}{Q+\eps^2\jjsumN a_j^2},
\]
where 
\[
N=\max\left\{m:\eps^2\sum_{j=0}^{m-1} a_j(a_m-a_j)<Q\right\}<+\infty.
\]
Consider $N$ defined above. Denote
\begin{multline*}
A_m=\eps^2\sum_{j=0}^{m-1} a_j(a_m-a_j)=\\=n^{\beta/r-1}\sigma^2\sum_{j=0}^{m-1} \sqrt{\left(1+\lambda_j^{2\beta/r}n^{2\beta/r}\right)\left(\lambda_m^{2\beta/r}n^{2\beta/r}-\lambda_j^{2\beta/r}n^{2\beta/r}\right)}.
\end{multline*}
Since the geometry condition on the graph is satisfied, for $j=i_0,\dots,\kappa n$ the eigenvalues of the graph Laplacian can be bounded in a following way
\[
C_1\Big(\frac j n\Big)^{2/r} \le \lambda_j \le C_2\Big(\frac j n\Big)^{2/r}.
\]
If $m<\kappa n$, then for some $K_1>0$
\[
A_m\geq n^{-1}\sigma^2 C_1\sum_{j=i_0}^{m-1} j^{\beta/r}\left(C_1 m^{\beta/r}-C_2 j^{\beta/r}\right)\geq K_1 n^{-1} m^{\beta/r}\sum_{j=i_0}^{m-1} j^{\beta/r}.
\]
Since $A_m$ is an increasing function of $m$, there exists $K_2>0$ such that for all 
\[
m>K_2 n^{r/(2\beta+r)}
\]
it holds that $A_m>Q$.
In a similar manner we can show that there exists $K_3>0$ such that for all 
\[
m<K_3 n^{r/(2\beta+r)}
\]
it holds that $A_m<Q$.

This leads us to the conclusion that $N\asymp n^{r/(2\beta+r)}$. Then equation (\ref{kappaeq}) has a unique solution that satisfies
\[
x=\frac {\sigma^2}{n\left(Q+\frac {\sigma^2}n \jjsumN a_j^2\right)}\jjsumN a_j \asymp\frac1n N^{1+\beta/r}\asymp n^{-\beta/(2\beta+r)}.
\]
\item Since graph $G$ satisfies the geometry assumption, we deduce from (i) that for some $K_4>0$
\[
l'_j\asymp \left(1-K_4 n^{-\beta/(2\beta+r)} j^{\beta/r}\right)_+, \text{ for } j= i_0,\dots, N.
\]
For $j=0,\dots, i_0 -1$ we bound $l'_j$ from above by $1$.
Then
\[
S\lle n^{-1}i_0 +n^{-1}\sum_{j=i_0}^{N-1} \left(1- K_4 n^{-\beta/(2\beta+r)} j^{\beta/r}\right)_+\lle n^{-1}N\lle n^{-2\beta/(2\beta+r)}.
\]
On the other hand, 
\[
S\gge n^{-1}\sum_{j=i_0}^{N-1} \left(1- K_4 n^{-\beta/(2\beta+r)} j^{\beta/r}\right)_+\gge n^{-2\beta/(2\beta+r)}.
\]
\item Note that for $j>N$ we have $v_j^2=0.$ We also know that $a_N<\frac{1}{x}.$ Then
\[
v_j^2 a_j^2=\frac{\sigma^2 a_j(1-x a_j)_+}{nx}\leq\frac{\sigma^2 a_N}{nx}\leq\frac{\sigma^2}{nx^2}.
\]
Hence, as $n\to\infty$,
\[
\max_{j=0,\dots,n-1}v_j^2 a_j^2=O\left(n^{-r/(2\beta+r)}\right).
\]
\end{enumerate}
This finishes the proof of the lemma.
\end{proof}

\subsubsection{Proof of the upper bound on the risk}
\label{perfect_test}

Recall that we only need to provide the upper bound in \eqref{eq: statement}. Consider the Pinsker estimator $\tilde f=(\l'_0 Z_0,\dots,l'_{n-1}Z_{n-1})$ with
 \[
l'_j=(1-x a_j)_{+},
\]
where $a_j^2=1+ \lambda_j^{2\beta/r}n^{2\beta/r}$ and $x$ is a unique solution of \eqref{kappaeq}. Using Lemma \ref{linear_minimax} and Lemma \ref{pinsker_const} we conclude that the Pinsker's estimator satisfies
\[
\sup_{f\in B_n(Q)} \EE_f \left(\sum_{i=0}^{n-1} (\tilde f_i-f_i)^2\right)\lle n^{-2\beta/(2\beta+r)}.
\]

\subsubsection{Proof of the lower bound on the risk}

We follow the usual steps of the general reduction scheme for obtaining minimax rates (see e.g. Chapter 3 of \cite{tsybakov} for details). First, we 
note that it is sufficient to only take into account the first $N$ coefficients in the decomposition of the target function, where 
\[
N=\max\left\{m:\eps^2\sum_{j=0}^{m-1} a_j(a_m-a_j)<Q\right\}
\]
with $a_j^2=1+ \lambda_j^{2\beta/r}n^{2\beta/r}$.
Indeed, if we 
denote the minimax risk by $R_n$, i.e.\
 \[
 R_n= \inf_{{\hat f}}\sup_{f\in B_n(Q)} \EE_f \left(\sum_{i=0}^{n-1} (\hat f_i-f_i)^2\right),
 \]
 and
 define
\begin{align*}
B_n(Q,N)=\{f^{(N)}= (f_0,\dots,f_{N-1},0,\dots,0)\in \RR^n:\sum_{j=0}^{N-1} a_j^2 f_j^2\leq Q\},
\end{align*}
then we have $B_n(Q,N)\subseteq B_n(Q)$ and hence
\begin{equation}
\label{parametric_ineq}
R_n\geq\inf_{\hat f^{(N)}\in B_n(Q,N)}\sup_{ f^{(N)}\in B_n(Q,N)}\EE_ f\jjsumN(\hat f_j- f_j)^2.
\end{equation}

Next we follow the usual step of bounding the minimax risk
by a Bayes risk. 
Consider the density $\mu( f^{(N)})=\prod_{j=0}^{N-1} \mu_{s_j} ( f_j)$ with respect to the Lebesgue measure on $\RR^N$. Here $s_j=(1-\delta)v_j^2$ for some $\delta\in(0,1)$ and $\mu_{\sigma}$ denotes the density of the Gaussian distribution with mean $0$ and variance $\sigma^2$. By (\ref{parametric_ineq}) we can bound the minimax risk from below by the Bayes risk 
\begin{equation}
\label{bayes_risk_ineq}
R_n\geq\inf_{\hat f^{(N)}\in B_n(Q,N)}\jjsumN\int_{B_n(Q,N)}\EE_ f(\hat f_j- f_j)^2\mu(f^{(N)})df^{(N)}\geq I^\star-r^\star,
\end{equation}
where
\begin{align*}
I^\star&=\inf_{\hat f^{(N)}\in B_n(Q,N)}\jjsumN\int_{\RR^N}\EE_ f (\hat f_j- f_j)^2\mu( f^{(N)})d f^{(N)};\\
r^\star&=\sup_{\hat f^{(N)}\in B_n(Q,N)}\jjsumN\int_{ B_n(Q,N)^c} \EE_ f (\hat f_j- f_j)^2\mu( f^{(N)})d f^{(N)}
\end{align*}
with $B(Q,N)^c=\RR^N\setminus B_n(Q,N).$ From the proof of Pinsker's theorem, see \cite{tsybakov}, we get the following bounds 
\begin{align*}
I^\star&\gge S,\\
r^\star&\lle \exp\left\{-K\left(\max_{j=0,\dots,n-1} v_j^2 a_j^2\right)^{-1}\right\}
\end{align*}
for some $K>0$. Using the results of Lemma \ref{pinsker_const} we conclude that $R_n\gge n^{-{2\beta}/{(2\beta+r)}}.$

\subsection{Proof of Theorem \ref{general_thm_cl}}
\label{proof_clas}

In order to prove the result in the classification case we use Fano's lemma and the usual general scheme for reducing a minimax estimation problem to a minimax testing problem.
(see for instance \cite{tsybakov} again).

\subsubsection{Proof of the upper bound on the risk}

We define the estimator that gives us an upper bound on the minimax risk based on the estimator $\tilde f$, which has been introduced in subsection \ref{perfect_test} of the proof of Theorem \ref{general_thm_reg}. Consider the estimator 
\[
\tilde \rho= \Psi\left(\sum_{i=0}^{n-1}\tilde f_i \psi_i\right).
\]
By the reasoning given in the aforementioned subsection and using the properties of the link function $\Psi$, we can see that
\begin{multline*}
\sup_{\rho\in \{\rho: \Psi^{-1}(\rho)\in H^{\beta}(Q)\}}\EE_{\rho}\|\tilde\rho-\rho\|_n^2\lle\\ \lle\sup_{f\in B_n(Q)} \EE_f \left(\sum_{i=0}^{n-1} (\tilde f_i-f_i)^2\right)\lle n^{-2\beta/(2\beta+r)}.
\end{multline*}

\subsubsection{Proof of the lower bound on the risk} 
The proof of the lower bound on the risk is based on a corollary of Fano's lemma, 
see, for instance, Corollary 2.6 in \cite{tsybakov}. Observe that by Markov's inequality for any soft label functions $\rho_0,\dots,\rho_M$ with $M\in\NN$ there exists $C>0$ such that
\begin{multline*}
\inf_{{\hat \rho}} \sup_{\rho\in \{\rho: \Psi^{-1}(\rho)\in H^{\beta}(Q)\}} \EE_\rho n^{2\beta/(2\beta+r)}\|{\hat \rho}-\rho\|^2_n\gge\\
\gge \inf_{{\hat \rho}} \max_{\rho\in\{\rho_0,\dots,\rho_M\}} \PP_{\rho} \left(\|\hat \rho - \rho\|^2_n\geq n^{-2\beta/(2\beta+r)}\right).
\end{multline*}
Consider probability measures $P_0, P_1, \dots, P_M$ corresponding to the soft label functions $\rho_0,\dots,\rho_M$. For a test $\mathscr{\phi}:\RR^n\to\{0,1,\dots,M\}$ define the average probability of error by
\[
{\bar{p}}_M (\phi)= \frac{1}{M+1}\sum_{j=0}^M P_j(\phi\neq j).
\]
Additionally, let 
\[
{\bar{p}}_M =\inf_{\phi} {\bar{p}}_M (\phi).
\]
From the general scheme for obtaining lower bounds for minimax risk (for more detail see Chapter 2 of \cite{tsybakov}) we know that if $\rho_0,\dots,\rho_M$ are such that for any pair $i,j\in\{0,\dots,M\}$ 
\begin{equation}
\label{eq: separation}
\|\rho_i-\rho_j\|_n\gge n^{-2\beta/(2\beta+r)}, \text{ when $i\neq j$},
\end{equation}
then
\[
\inf_{{\hat \rho}} \max_{\rho\in\{\rho_0,\dots,\rho_M\}} \PP_{\rho} \left(\|\hat \rho - \rho\|^2_n\geq C n^{-2\beta/(2\beta+r)}\right)\gge \bar p_M.
\]
Also, Corollary 2.6 in \cite{tsybakov} states that if $P_0, P_1, \dots, P_M$ satisfy 
\begin{equation}
\label{bound_on_P}
\frac{1}{M+1} \sum_{j=1}^M K(P_j, P_0)\leq \alpha \log M,
\end{equation}
for some $0<\alpha<1$, then 
\[
\bar p_M\geq\frac{\log(M+1)-\log 2}{\log M}-\alpha.
\]
Here $K(\cdot,\cdot)$ is the Kullback--Leibler divergence.
Hence, if we construct the probability measures $P_0, P_1, \dots, P_M$ corresponding to some soft label functions $\rho_0,\dots,\rho_M$ for which \eqref{eq: separation} and \eqref{bound_on_P} hold, we will have
\[
\inf_{{\hat \rho}} \sup_{\rho\in \{\rho: \Psi^{-1}(\rho)\in H^{\beta}(Q)\}} \EE_\rho n^{2\beta/(2\beta+r)}\|{\hat \rho}-\rho\|^2_n\gge 
\frac{\log(M+1)-\log 2}{\log M}-\alpha.
\]
If $M\to\infty$, as $n\to\infty$, the required result follows.

Let $N=n^{r/(2\beta+r)}$ and let $\psi_0,\dots,\psi_{n-1}$ to be an orthonormal eigenbasis of the graph Laplacian $L$ with respect to the $\|\cdot\|_n$--norm. For $\delta>0$ and $\theta=(\theta_0,\dots\theta_{N-1})\in\{\pm1\}^N$ define
\[
f_\theta=\delta N^{-{(2\beta+r)}/{2r}}\jjsumN \theta_j\psi_j.
\]
We will select $M$ vectors of coefficients $\theta^{(j)}$ such that the probability measures corresponding to $\rho_j= \Psi(f_{\theta^{(j)}})$ will satisfy \eqref{bound_on_P}, where $\Psi$ is the link function.

Observe that for small enough $\delta>0$ functions $f_\theta$ belong to the class $H^{\beta}(Q)$. Indeed, using the geometry assumption we obtain
\begin{multline*}
\qv{f_{\theta}, (I+(n^\frac{2}{r} L)^\beta) f_{\theta}}_n=\delta^2N^{-{(2\beta+r)}/{r}}\jjsumN(1+n^{2\beta/r}\lambda_j^\beta)\leq\\\leq\delta^2N^{-{(2\beta+r)}/{r}}\left(N+C_2 i_0^{(2\beta+r)/r}+C_2\sum_{j=i_0}^{N-1}j^{2\beta/r}\right)\leq K_1\delta^2
\end{multline*}
for some constant $K_1>0.$ 

We pick a subset $\{\theta^{(1)},\dots,\theta^{(M)}\}$ of $\{\pm1\}^N$ such that for any pair $i,j\in\{1,\dots,M\}$ such that $i\neq j$ the vectors from the subset were sufficiently distant from each other
\begin{equation}
\label{eq: hamming}
d_h(\theta^{(i)},\theta^{(j)})\gge N,
\end{equation}
where $d_h(\theta,\theta')=\jjsumN \mathds{1}_{\theta_j=\theta'_j}$ is the Hamming distance.
By the Varshamov--Gilbert bound (see for example Lemma 2.9 in \cite{tsybakov}) we know that there exist such a subset $\{\theta^{(1)},\dots,\theta^{(M)}\}$, and the size $M$ of this subset satisfies
\[
M\geq b^{N}\]
 for some $1<b<2$. Let $\theta^{(0)}=(0,\dots,0)\in \RR^N$. We define a set of probability measures $\{P_0,\dots,P_M\}$ by setting $P_j=P_{\rho_j}$, where $\rho_j= \Psi(f_{\theta^{(j)}})$.

In order to show that the $P_j$ satisfy \eqref{bound_on_P} we present a technical lemma. In the classification setting the Kullback--Leibler divergence $K(\cdot,\cdot)$ satisfies
\[
K(P_\rho,P_\rho')=\sum\limits_{i=1}^n{\left(\rho(i) \log\frac{\rho(i)}{\rho'(i)}-(1-\rho(i))\log\frac{1-\rho(i)}{1-\rho'(i)}\right)}.
\]

\begin{lem}
\label{distances}
If $\frac{\Psi'}{\Psi(1-\Psi)}$ is bounded, then there exists $c>0$ such that for any $v_1, v_2 \in\RR^n$ we have
\[
K(P_{\Psi(v_1)},P_{\Psi(v_2)})\leq n c||v_1-v_2||^2_n. 
\]
\end{lem}
\begin{proof}
For every $x\in\RR$ consider the function $g_x:\RR \to \RR$ defined as
\[
g_x(y)=\Psi(x)\log\frac{\Psi(x)}{\Psi(y)}+(1-\Psi(x))\log\frac{1-\Psi(x)}{1-\Psi(y)}. 
\]
We see that $g'_x(y)=\frac{\Psi'(y)}{\Psi(y)(1-\Psi(y))}(\Psi(y)-\Psi(x)). $ Then by Taylor's theorem we can see that
\[
|g_x(y)|\leq\sup\limits_{v\in[x, y]\cup[y, x]}{\left|\frac{\Psi'(v)}{\Psi(v)(1-\Psi(v))}\right|} \sup\limits_{v\in[x, y]\cup[y, x]}{\left|\Psi'(v)\right|}(x-y)^2. 
\]
The statement of the lemma follows.
\end{proof} 
By Lemma \ref{distances}, we obtain for some $K_2>0$
\[
K(P_j,P_0)\leq K_2 n \|f_{\theta^{(j)}}- 0\|_n^2 =4K_2\delta^2 n N^{-2\beta/r},
\]
since
\begin{equation}
\label{eq: nnorm}
\|f_\theta-f_{\theta'}\|_n^2=4\delta^2N^{-(2\beta+r)/r}d_{h}(\theta,\theta').
\end{equation}
Observe that this bound does not depend on $j$. Hence,
\[
\frac{1}{M+1} \sum_{j=1}^M K(P_j, P_0)\leq K_2\delta^2 n N^{-2\beta/r}=K_2\delta^2\log M.
\]
We can choose $\delta>0$ to be small enough such that the condition \eqref{bound_on_P} is satisfied with some $0<\alpha<1$.

To finish the proof of the theorem we need to show that $\rho_0,\dots,\rho_M$ satisfy \eqref{eq: separation}. From \eqref{eq: hamming} and \eqref{eq: nnorm} we get
\[
 \|f_{\theta^{(i)}}-f_{\theta^{(j)}}\|_n\gge n^{-\beta/(2\beta+r)}.
\]
Moreover, by the assumption of the theorem we have for any $j=1,\dots,M$
\[
\max_{i=1,\dots,n} |f_{\theta^{(j)}}(i)| \lle N^{1-(2\beta+r)/2r}\max_{i=1,\dots,n}|\psi_j(i)|\lle N^{(r-2\beta)/2r}.
\]
For $\beta\geq r/2$ the norm is then bounded by some constant, which does not depend on $n$ or $j$. Hence, there exists $K_3\geq 0$ such that for every $i=1,\dots,n$ and every $j=1,\dots, M$
\[
|f_{\theta^{(j)}}(i)|\leq K_3.
\] 
Observe that since $\Psi'(x)\neq0$ for any $x\in\RR$, there exists $K_4>0$ such that for any $x, y\in[-K_3, K_3]$ 
\[
|\Psi(x)-\Psi(y)|\geq K_4|x-y|.
\]
Thus, for any pair $i,j\in\{1,\dots,M\}$ such that $i\neq j$
\[
\|\rho_i-\rho_j\|_n=\|\Psi(f_{\theta^{(i)}})-\Psi(f_{\theta^{(i)}})\|_n\gge n^{-\beta/(2\beta+r)}.
\]
This completes the proof of the theorem.

\newpage

\bibliographystyle{harry}
\bibliography{pathgraph}
\end{document}